\documentclass[fleqn, pdftex]{article} %arXiv

\usepackage{amsmath,amssymb,amsthm}

\title{On inclusions between quantified provability logics}
\author{Taishi Kurahashi\footnote{Email: kurahashi@people.kobe-u.ac.jp}
\footnote{Graduate School of System Informatics, Kobe University, 1-1 Rokkodai, Nada, Kobe 657-8501, Japan.}}
\date{}

\theoremstyle{plain}
\newtheorem{thm}{Theorem}[section]
\newtheorem{lem}[thm]{Lemma}
\newtheorem{prop}[thm]{Proposition}
\newtheorem{cor}[thm]{Corollary}
\newtheorem{fact}[thm]{Fact}
\newtheorem{prob}[thm]{Problem}

\theoremstyle{definition}
\newtheorem{defn}[thm]{Definition}

\newtheorem{rem}[thm]{Remark}

\newcommand{\PA}{\mathbf{PA}}
\newcommand{\IS}{\mathbf{I\Sigma_1}}
\newcommand{\GL}{\mathbf{GL}}

\newcommand{\N}{\mathbb{N}}

\newcommand{\PR}{\mathrm{Pr}}
\newcommand{\Con}{\mathrm{Con}}
\newcommand{\D}{\mathrm{D}}
\newcommand{\gdl}[1]{\ulcorner#1\urcorner}

\newcommand{\PL}{\mathsf{PL}}
\newcommand{\QPL}{\mathsf{QPL}}
\newcommand{\Th}{\mathsf{Th}}

\begin{document}

\maketitle

\begin{abstract}
We investigate several consequences of inclusion relations between quantified provability logics. 
Moreover, we give a necessary and sufficient condition for the inclusion relation between quantified provability logics with respect to $\Sigma_1$ arithmetical interpretations. 
\end{abstract}

\section{Introduction}\label{Sec:Intro}

The notion of provability is a kind of modality, and modal logical studies of formalized provability have been extensively proceeded by many authors. 
Such studies have had many successes, especially in the framework of propositional modal logic. 
Solovay's arithmetical completeness theorem \cite{Sol76} is one of them. 
For every recursively enumerable extension $T$ of Peano Arithmetic $\PA$, let $\PR_T(x)$ be a usual provability predicate of $T$. 
A \textit{$T$-arithmetical interpretation} is a mapping $f_T$ from the set of all propositional modal formulas to the set of sentences of arithmetic such that $f_T$ commutes with each propositional connective and $f_T$ maps $\Box A$ to $\PR_T(\gdl{f_T(A)})$. 
Let $\PL(T)$ be the set of all propositional modal formulas $A$ such that $T \vdash f_T(A)$ for every $T$-arithmetical interpretation $f_T$. 
This set is called the \textit{propositional provability logic} of $T$.  
Solovay's arithmetical completeness theorem states that if $T$ is a $\Sigma_1$-sound recursively enumerable extension of $\PA$, then $\PL(T)$ is exactly the propositional modal logic $\GL$.
Thus $\PL(T)$ is recursive, but does not contain any elements specific to the theory $T$.

Formalized provability is also studied in the framework of quantified modal logic. 
The main target of this study is the \textit{quantified provability logic $\QPL(T)$} of $T$, which consists of quantified modal sentences verifiable in $T$ under any $T$-arithmetical interpretation. 
Boolos \cite{Boo79} asked if $\QPL(\PA)$ is recursively enumerable or not, and in contrast to the propositional case, Vardanyan \cite{Var86} proved that $\QPL(\PA)$ is $\Pi^0_2$-complete. 
Hence the analogue of Solovay's arithmetical completeness theorem never holds in the case of quantified modal logic.
Moreover, Montagna \cite{Mon84} showed that some results which hold in the case of propositional logic are not inherited in the quantified case. 
Among other things, he proved that $\QPL(\PA)$ is not a subset of $\QPL(\mathbf{BG})$, where $\mathbf{BG}$ is the Bernays--G\"odel set theory.
Thus $\QPL(T)$ can vary depending on the theory $T$. 

Artemov \cite{Art86} showed that the quantified provability logic $\QPL(T)$ of $T$ can be different depending on the choice of a formula defining $T$. 
More precisely, we say that a formula $\tau(v)$ is a \textit{definition} of a theory $T$ if for any natural number $n$, $\tau(\overline{n})$ is true if and only if $n$ is the G\"odel number of an axiom of $T$. 
For each $\Sigma_1$ definition $\tau(v)$ of $T$, we can construct a $\Sigma_1$ provability predicate $\PR_\tau(x)$ of $T$ saying that ``$x$ is (the G\"odel number of a formula) provable in the theory defined by $\tau(v)$''. 
The notion of $\tau$-arithmetical interpretations is introduced as well by using $\PR_{\tau}(x)$ instead of $\PR_T(x)$. 
Then, the quantified provability logic $\QPL_\tau(T)$ of $\tau(v)$ is defined to be the set of all quantified modal sentences provable in $T$ under all $\tau$-arithmetical interpretations. 
Artemov proved that for any $\Sigma_1$-sound recursively enumerable extension $T$ of $\PA$ and any $\Sigma_1$ definition $\tau_0(v)$ of $T$, there exists a $\Sigma_1$ definition $\tau_1(v)$ of $T$ such that $\QPL_{\tau_0}(T) \nsubseteq \QPL_{\tau_1}(T)$. 

The results of Montagna and Artemov seem to indicate that inclusion relations between quantified provability logics are rarely established. 
Indeed, Kurahashi \cite{Kur13B} proved that for any natural numbers $i$ and $j$ with $0 < i < j$, there exists a $\Sigma_1$ definition $\sigma_i(v)$ of the theory $\mathbf{I\Sigma_i}$ such that for all $\Sigma_1$ definitions $\sigma_j(v)$ of $\mathbf{I\Sigma_j}$, $\QPL_{\sigma_i}(\mathbf{I\Sigma_i}) \nsubseteq \QPL_{\sigma_j}(\mathbf{I\Sigma_j})$ and $\QPL_{\sigma_j}(\mathbf{I\Sigma_j}) \nsubseteq \QPL_{\sigma_i}(\mathbf{I\Sigma_i})$. 
The situation of the inclusion relation between quantified provability logics is completely different from that of propositional case: it is known that for any theories $T_0$ and $T_1$, at least one of $\PL(T_0) \subseteq \PL(T_1)$ and $\PL(T_1) \subseteq \PL(T_0)$ holds (cf.~Visser \cite{Vis84}).

From this point of view, in the present paper, we investigate several consequences of the inclusion $\QPL_{\tau_0}(T_0) \subseteq \QPL_{\tau_1}(T_1)$ between quantified provability logics. 
Among other things, we prove that if $\QPL_{\tau_0}(T_0) \subseteq \QPL_{\tau_1}(T_1)$, then
\begin{enumerate}
	\item $T_0 + \Con_{\tau_0}$ is a subtheory of $T_1 + \Con_{\tau_1}$; 
	\item $T_0$ is $\Sigma_1$-conservative over $T_1$; 
	\item $\Con_{\tau_0}$ and $\Con_{\tau_1}$ are provably equivalent over $T_1$; and
	\item For any formula $\varphi(\vec{x})$, 
\[
	T_1 \vdash \forall \vec{x} \left (\PR_{\tau_0}(\gdl{\Con_{\tau_0} \to \varphi(\vec{\dot{x}})}) \leftrightarrow \PR_{\tau_1}(\gdl{\Con_{\tau_1} \to \varphi(\vec{\dot{x}})}) \right). 
\]
\end{enumerate}
Thus from our results, we certify that the inclusion relation between quantified provability logics holds only under limited situations. 
Moreover, our results also show that the quantified provability logic $\QPL_{\tau}(T)$ is not only complex, but also possesses much information about the theory $T$ and the provability predicate $\PR_\tau(x)$.

We also investigate provability logics with respect to $\Sigma_1$ arithmetical interpretations. 
In the propositional case, a $T$-arithmetical interpretation $f_T$ is called $\Sigma_1$ if for any propositional variable $p$, $f_T(p)$ is a $\Sigma_1$ sentence. 
Let $\PL^{\Sigma_1}(T)$ be the set of all propositional modal formulas $A$ such that $T \vdash f_T(A)$ for every $T$-arithmetical interpretation $f_T$ which is $\Sigma_1$. 
Visser proved that $\PL^{\Sigma_1}(\PA)$ is also recursive and exactly the propositional modal logic $\mathbf{GLV}$ (see Boolos \cite{Boo93}).
In the quantified case, Berarducci \cite{Ber89} also proved that $\QPL^{\Sigma_1}(\PA)$ is $\Pi^0_2$-complete. 
Thus, the situations of $\Sigma_1$ provability logics do not seem to be different from those of usual provability logics. 

On the other hand, there is an advantage to dealing with $\Sigma_1$ arithmetical interpretations for our purposes, which allows us to improve Artemov's Lemma used in the proof of Vardanyan's theorem.
Then, we can give a necessary and sufficient condition for the inclusion relation between quantified provability logics with respect to $\Sigma_1$ arithmetical interpretations. 
Namely, we prove that $\QPL_{\tau_0}^{\Sigma_1}(T_0) \subseteq \QPL_{\tau_1}^{\Sigma_1}(T_1)$ if and only if $T_0$ is a subtheory of $T_1$ and for any formula $\varphi(\vec{x})$, $T_1 \vdash \forall \vec{x} (\PR_{\tau_0}(\gdl{\varphi(\vec{\dot{x}})}) \leftrightarrow \PR_{\tau_1}(\gdl{\varphi(\vec{\dot{x}})}))$.

\section{Preliminaries}

Let $\mathcal{L}_A = \{0, S, +, \times, <, =\}$ be the language of first-order arithmetic. 
We call a set of $\mathcal{L}_A$-sentences simply a \textit{theory}. 
Peano Arithmetic $\PA$ is the theory consisting of basic axioms for $\mathcal{L}_A$ and induction axioms for $\mathcal{L}_A$-formulas. 
$\IS$ is the theory obtained from $\PA$ by restricting induction axioms to $\Sigma_1$ formulas. 
Throughout the present paper, $T$, $T_0$ and $T_1$ always denote recursively enumerable extensions of $\IS$\footnote{Based on the result of de Jonge \cite{DeJ05} that Artemov's Lemma (Fact \ref{AL}) holds for the theory $\IS$, we adopted $\IS$ as the base theory in this paper. 
See the paragraph immediately following Fact \ref{BooSurj}.}.
In the present paper. 
Let $\Th(T)$ be the set of all $\mathcal{L}_A$-sentences provable in $T$. 
Also, for each class $\Gamma$ of formulas, let $\Th_{\Gamma}(T) : = \Th(T) \cap \Gamma$. 
The standard model of arithmetic is denoted by $\N$. 
We say that $T$ is \textit{$\Sigma_1$-sound} if every element of $\Th_{\Sigma_1}(T)$ is true in $\N$. 
Notice that $\Sigma_1$-soundness implies consistency. 

For each natural number $n$, the numeral for $n$ is denoted by $\overline{n}$. 
We fix some natural G\"odel numbering, and for each $\mathcal{L}_A$-formula $\varphi$, let $\gdl{\varphi}$ be the numeral for the G\"odel number of $\varphi$. 
We say a formula $\tau(v)$ is a \textit{definition} of a theory $T$ if for any natural number $n$, $\N \models \tau(\overline{n})$ if and only if $n$ is the G\"odel number of some axiom of $T$. 
Hereafter, we assume that $\tau(v)$, $\tau_0(v)$ and $\tau_1(v)$ always denote $\Sigma_1$ definitions of $T$, $T_0$ and $T_1$, respectively. 
Then, we can construct a $\Sigma_1$ formula $\PR_{\tau}(x)$ saying that ``$x$ is (the G\"odel number of a formula) provable in the theory defined by $\tau(v)$''. 
The following fact is well-known. 

\begin{fact}[Derivability conditions (see Boolos \cite{Boo93} and Lindstr\"om \cite{Lin03})]\label{DC}
For any formulas $\varphi(\vec{x})$ and $\psi(\vec{x})$, 
\begin{enumerate}
	\item If $T \vdash \varphi(\vec{x})$, then $\IS \vdash \PR_\tau(\gdl{\varphi(\vec{\dot{x}})})$; 
	\item $\IS \vdash \PR_\tau(\gdl{\varphi(\vec{\dot{x}}) \to \psi(\vec{\dot{x}})}) \to (\PR_\tau(\gdl{\varphi(\vec{\dot{x}})}) \to \PR_\tau(\gdl{\psi(\vec{\dot{x}})}))$; 
	\item If $\varphi(\vec{x})$ is a $\Sigma_1$ formula, then $\IS \vdash \varphi(\vec{x}) \to \PR_\tau(\gdl{\varphi(\vec{\dot{x}})})$. 
\end{enumerate}
\qed
\end{fact}
Here $\gdl{\varphi(\vec{\dot{x}})}$ is an abbreviation for $\gdl{\varphi(\dot{x_1}, \ldots, \dot{x_n})}$ that is a primitive recursive term corresponding to a primitive recursive function calculating the G\"odel number of $\varphi(\overline{k_1}, \ldots, \overline{k_n})$ from $k_1, \ldots, k_n$. 

Let $\Con_\tau$ be the $\Pi_1$ sentence $\neg \PR_{\tau}(\gdl{0 = \overline{1}})$ stating that the theory defined by $\tau(v)$ is consistent. 
%$\PR_T(x)$ and $\Con_T$ denote $\PR_\tau(x)$ and $\Con_\tau$ for some natural $\Sigma_1$ definition $\tau(v)$ of $T$, respectively. 
For each sentence $\varphi$, let $(\tau + \varphi)(v)$ be the $\Sigma_1$ definition $\tau(v) \lor v = \gdl{\varphi}$ of $T + \varphi$. 
Then it is known that the formalized version of the deduction theorem holds: $\IS \vdash \forall x (\PR_{\tau + \varphi}(x) \leftrightarrow \PR_\tau(\gdl{\varphi} \dot{\to} x))$. 
Here $u \dot{\to} v$ is a primitive recursive term corresponding to a primitive recursive function calculating the G\"odel number of $\varphi \to \psi$ from the G\"odel numbers of $\varphi$ and $\psi$. 

The language of quantified modal logic is the language of first-order predicate logic without function and constant symbols equipped with the unary modal operators $\Box$ and $\Diamond$. 
%We define translations of quantified modal formulas into $\mathcal{L}_A$-formulas. 
We may assume that the languages of quantified modal logic and first-order arithmetic have the same variables. 

\begin{defn}
A mapping $f$ from the set of all atomic formulas of quantified modal logic to the set of $\mathcal{L}_A$-formulas satisfying the following condition is called an \textit{arithmetical interpretation}: For each atomic formula $P(x_1, \ldots, x_n)$, $f(P(x_1, \ldots, x_n))$ is an $\mathcal{L}_A$-formula $\varphi(x_1, \ldots, x_n)$ with the same free variables, and moreover $f(P(y_1, \ldots, y_n))$ is $\varphi(y_1, \ldots, y_n)$ for any variables $y_1, \ldots, y_n$. 
\end{defn}

\begin{defn}
Each arithmetical interpretation $f$ is uniquely extended to a mapping $f_{\tau}$ from the set of all quantified modal formulas to the set of $\mathcal{L}_A$-formulas inductively as follows: 
\begin{enumerate} 
	\item $f_\tau(\bot)$ is $0 = \overline{1}$;
	\item $f_\tau$ commutes with each propositional connective and quantifier; 
	\item $f_\tau(\Box A(x_1, \ldots, x_n))$ is the formula $\PR_\tau(\gdl{f_\tau(A(\dot{x_1}, \ldots, \dot{x_n}))})$. 
\end{enumerate}
\end{defn}

Notice that any quantified modal formula $A$ has the same free variables as $f_\tau(A)$. 
We are ready to introduce the quantified provability logic of $\tau(v)$.

\begin{defn}
The \textit{quantified provability logic $\QPL_\tau(T)$ of $\tau(v)$} is the set \[
	\{A \mid A\ \text{is a sentence and for all arithmetical interpretations}\ f, T \vdash f_\tau(A)\}.
\]
\end{defn}

The main purpose of the present paper is to investigate the inclusion relation $\QPL_{\tau_0}(T_0) \subseteq \QPL_{\tau_1}(T_1)$ between quantified provability logics. 
For this purpose, we heavily use Artemov's Lemma (Fact \ref{AL}) that is used in the proof of Vardanyan's theorem on the $\Pi_2$-completeness of the quantified provability logic of $\PA$. 
To state Artemov's Lemma, we prepare some definitions.

\begin{defn}
We prepare predicate symbols $P_Z(x)$, $P_S(x, y)$, $P_A(x, y, z)$, $P_M(x, y, z)$, $P_L(x, y)$ and $P_E(x, y)$ corresponding to members $0$, $S$, $+$, $\times$, $<$ and $=$ of $\mathcal{L}_A$, respectively. 
For each $\mathcal{L}_A$-formula $\varphi$, let $\varphi^\ast$ be a logically equivalent $\mathcal{L}_A$-formula where each atomic formula is one of the forms $x = 0$, $S(x) = y$, $x + y = z$, $x \times y = z$, $x < y$ and $x = y$. 
Let $\varphi^\circ$ be a relational formula obtained from $\varphi^\ast$ by replacing each atomic formula with the corresponding relation symbol in $\{P_Z, P_S, P_A, P_M, P_L, P_E\}$ adequately. 
Then $\varphi^\circ$ is a quantified modal formula. 
\end{defn}

Let $\mathrm{Seq}(s)$ be the formula naturally expressing that ``$s$ is a finite sequence''. 
Also let $lh(s)$ and $(s)_x$ be primitive recursive terms corresponding to primitive recursive functions calculating the length and $x$-th component of a finite sequence $s$, respectively. 

\begin{defn}
For each arithmetical interpretation $f$, let $R_f(x, y)$ be the formula
\[
	\exists s(\mathrm{Seq}(s) \land lh(s) = x + 1 \land (s)_x = y \land f(P_Z((s)_0)) \land \forall z < x\, f(P_S((s)_z, (s)_{z+1}))). 
\]
Let $R_f(\vec{x}, \vec{y})$ denote a conjunction $R_f(x_0, y_0) \land R_f(x_1, y_1) \land \cdots \land R_f(x_n, y_n)$. 
\end{defn}

The formula $R_f(x, y)$ means that $y$ represents $x$ under the interpretation that $f(P_Z(u))$ and $f(P_S(u, v))$ say ``$u$ represents $0$'' and ``$v$ represents the successor of a number represented by $u$'', respectively. 

We introduce the modal sentence $\D$ asserting the completeness of $P_K$ and $\neg P_K$ for every newly introduced predicate symbol $P_K$. 

\begin{defn}
Let $\D$ be the modal sentence
\[
	\bigwedge_{K \in \{Z, S, A, M, L, E\}} \Bigl( \forall \vec{x}(P_K(\vec{x}) \to \Box P_K(\vec{x})) \land \forall \vec{x}(\neg P_K(\vec{x}) \to \Box \neg P_K(\vec{x})) \Bigr).
\]
\end{defn}

We are ready to state Artemov's Lemma. 
In the statement of the lemma, the $\mathcal{L}_A$-sentence $\chi$ is a conjunction of several basic sentences of arithmetic such as $\forall x \exists y (S(x) = y)$ and $\forall x(x + 0 = x)$, which serves to incorporate a structure of arithmetic into a set. 

\begin{fact}[Artemov's Lemma (see {\cite[p.232]{Boo93}})]\label{AL}
There exists an $\mathcal{L}_A$-sentence $\chi$ such that $\IS \vdash \chi$ and for any arithmetical interpretation $f$ and $\mathcal{L}_A$-formula $\varphi(\vec{x})$, 
\[
	\IS \vdash \Con_\tau \land f_\tau(\D) \land f_\tau(\chi^\circ) \land R_f(\vec{x}, \vec{y}) \to \Bigl(\varphi(\vec{x}) \leftrightarrow f_\tau(\varphi^\circ(\vec{y}))\Bigr). 
\]
\qed
\end{fact}

We give a short outline of a proof of Artemov's Lemma based on the presentation in \cite{Kur13A}. 
Let $M$ be a model of $\IS + \Con_\tau \land f_\tau(\D) \land f_\tau(\chi^\circ)$. 
By the aid of $f_\tau(\chi^\circ)$, $f_\tau(P_E(x, y))$ defines an equivalence relation $\sim$ on $M$. 
Let $[a]$ be the equivalence class of $a \in M$ with respect to $\sim$. 
Then, the relations on $M$ defined by the formulas $P_K(\vec{x})$ for $K \in \{Z, S, A, M, L\}$ induce an $\mathcal{L}_A$-structure $M_f$ with the domain $\{[a] \mid a \in M\}$. 
For instance, $M_f \models [a] + [b] = [c] \iff M \models f_\tau(P_A(a, b, c))$. 
The sentence $f_\tau(\chi^\circ)$ guarantees that $M_f$ is well-defined and indeed an $\mathcal{L}_A$-structure satisfying a sufficiently strong fragment of $\IS$, and that for any $\vec{a} \in M$, $M_f \models \varphi(\vec{[a]}) \iff M \models f_\tau(\varphi^\circ(\vec{a}))$. 
Also $M$ is isomorphic to an initial segment of $M_f$ via an embedding defined by the formula $R_f(x, y)$. 
Moreover, from the sentence $\Con_\tau \land f_\tau(\D)$, we obtain the equivalences
\[
	f_\tau(P_K(\vec{x})) \leftrightarrow \PR_\tau(\gdl{f_\tau(P_K(\vec{\dot{x}}))})\ \text{and}\ \neg f_\tau(P_K(\vec{x})) \leftrightarrow \PR_\tau(\gdl{\neg f_\tau(P_K(\vec{\dot{x}}))}) 
\]
in $M$ for each $K \in \{Z, S, A, M, L, E\}$. 
Then both $f_\tau(P_K(\vec{x}))$ and $\neg f_\tau(P_K(\vec{x}))$ are equivalent to $\Sigma_1$ formulas in $M$. 
By applying a proof of Tennenbaum's theorem (see Kaye \cite{Kay91}), we obtain that $M$ and $M_f$ are in fact isomorphic, and hence are elementarily equivalent. 
Therefore, if $M \models R_f(\vec{a}, \vec{b})$, then $M \models \varphi(\vec{a})$ is equivalent to $M_f \models \varphi(\vec{[b]})$. 
Hence $M \models \varphi(\vec{a}) \leftrightarrow f_\tau(\varphi^\circ(\vec{b}))$. 

In the proof of Artemov's Lemma, the following facts are also used.  

\begin{fact}[See Boolos {\cite[Lemma 17.6]{Boo93}}]\label{BooSigma_1}
For any $\Sigma_1$ formula $\varphi(\vec{x})$ and arithmetical interpretation $f$, 
\[
	\IS \vdash f_\tau(\chi^\circ) \land R_f(\vec{x}, \vec{y}) \to \Bigl(\varphi(\vec{x}) \to f_\tau(\varphi^\circ(\vec{y})) \Bigr). 
\]
\qed
\end{fact}

\begin{fact}[See Boolos {\cite[Lemma 17.8]{Boo93}}]\label{BooSurj}
For any arithmetical interpretation $f$, 
\[
	\IS \vdash \Con_\tau \land f_\tau(\D) \land f_\tau(\chi^\circ) \to \forall y \exists x R_f(x, y). 
\]
\qed
\end{fact}

Facts \ref{BooSigma_1} and \ref{BooSurj} follow from the observations that $M$ is isomorphic to an initial segment of $M_f$ and $R_f(x, y)$ defines a surjection from $M$ onto $M_f$, respectively. 
In Boolos \cite{Boo93}, these facts including Artemov's Lemma are stated in the forms that the corresponding formulas are proved in $\PA$, and de Jonge \cite{DeJ05} proved that $\PA$ can be replaced by $\IS$ (see also \cite{Kur13A}).

\begin{defn}
An arithmetical interpretation $f$ is \textit{natural} if for each $K \in \{Z, S, A, M, L, E\}$, $f$ maps $P_K(\vec{x})$ to the intended atomic formula (for example, $f(P_A(x, y, z))$ is $x + y = z$). 
\end{defn}

For every quantified modal formula $A$, let $\boxdot A$ be an abbreviation for $A \land \Box A$. 

\begin{prop}\label{Natural}
Let $f$ be any natural arithmetical interpretation. 
\begin{enumerate}
	\item For any $\mathcal{L}_A$-formula $\varphi(\vec{x})$, $\IS \vdash \forall \vec{x}(f_\tau(\varphi^\circ(\vec{x})) \leftrightarrow \varphi(\vec{x}))$;  
	\item $\IS \vdash f_\tau(\boxdot \D) \land f_\tau(\boxdot \chi^\circ)$. 
\end{enumerate}
\end{prop}
\begin{proof}
1. By induction on the construction of $\varphi(\vec{x})$. 

2. For each $K \in \{Z, S, A, M, L, E\}$, since $f_\tau(P_K(\vec{x}))$ is $\Delta_0$, it follows from Fact \ref{DC}.3 that $\IS$ proves $f_\tau(P_K(\vec{x})) \to \PR_\tau(\ulcorner f_\tau(P_K(\vec{\dot{x}})) \urcorner)$ and $\neg f_\tau(P_K(\vec{x})) \to \PR_\tau(\ulcorner \neg f_\tau(P_K(\vec{\dot{x}})) \urcorner)$. 
Thus $\IS \vdash f_\tau(\D)$. 
By Fact \ref{DC}.1, $\IS \vdash \PR_{\tau}(\gdl{f_{\tau}(\D)})$, and hence $\IS \vdash f_{\tau}(\boxdot \D)$. 

Also by Clause 1, $\IS \vdash f_\tau(\chi^\circ) \leftrightarrow \chi$. 
Since $\IS$ proves $\chi$, $\IS \vdash f_\tau(\chi^\circ)$. 
As above, $\IS \vdash f_\tau(\boxdot \chi^\circ)$ also holds. 
\end{proof}

Artemov's Lemma is used to prove Vardanyan's theorem, but what is important to us is the following observation by Visser and de Jonge.

\begin{fact}[Visser and de Jonge {\cite[Theorem 3]{VD06}}]\label{VD}
For any $\mathcal{L}_A$-sentence $\varphi$, the following are equivalent: 
\begin{enumerate}
	\item $T + \Con_\tau \vdash \varphi$. 
	\item $\Diamond \top \land \D \land \chi^\circ \to \varphi^\circ \in \QPL_\tau(T)$. 
\end{enumerate}
\end{fact}
We give a proof of Visser and de Jonge's fact. 
\begin{proof}
$(1 \Rightarrow 2)$: 
Suppose $T + \Con_\tau \vdash \varphi$. 
By Artemov's Lemma, for any arithmetical interpretation $f$, 
\[
	T \vdash \Con_\tau \land f_\tau(\D) \land f_\tau(\chi^\circ) \to f_\tau(\varphi^\circ). 
\]
Thus $T \vdash f_\tau(\Diamond \top \land \D \land \chi^\circ \to \varphi^\circ)$. 
Hence $\Diamond \top \land \D \land \chi^\circ \to \varphi^\circ \in \QPL_\tau(T)$. 

$(2 \Rightarrow 1)$: 
Suppose $\Diamond \top \land \D \land \chi^\circ \to \varphi^\circ \in \QPL_\tau(T)$. 
For a natural arithmetical interpretation $f$, 
\[
	T \vdash \Con_\tau \land f_\tau(\D) \land f_\tau(\chi^\circ) \to f_\tau(\varphi^\circ). 
\]
By Proposition \ref{Natural}, $T + \Con_\tau \vdash \varphi$. 
\end{proof}

Visser and de Jonge's fact states that $\QPL_\tau(T)$ has the complete information about $\Th(T + \Con_\tau)$. 
Then we obtain some corollaries concerning inclusions between quantified provability logics. 

\begin{cor}\label{VDCor1}\leavevmode
\begin{enumerate}
	\item If $\QPL_{\tau_0}(T_0) \subseteq \QPL_{\tau_1}(T_1)$, then $\Th(T_0 + \Con_{\tau_0}) \subseteq \Th(T_1 + \Con_{\tau_1})$; 
	\item If $\QPL_{\tau_0}(T_0) = \QPL_{\tau_1}(T_1)$, then $\Th(T_0 + \Con_{\tau_0}) = \Th(T_1 + \Con_{\tau_1})$. 
\end{enumerate}
\end{cor}
\begin{proof}
1. Suppose $\QPL_{\tau_0}(T_0) \subseteq \QPL_{\tau_1}(T_1)$. 
Let $\varphi$ be any $\mathcal{L}_A$-sentence with $T_0 + \Con_{\tau_0} \vdash \varphi$. 
Then from Fact \ref{VD}, $\Diamond \top \land \D \land \chi^\circ \to \varphi^\circ \in \QPL_{\tau_0}(T_0)$. 
By the supposition, $\Diamond \top \land \D \land \chi^\circ \to \varphi^\circ \in \QPL_{\tau_1}(T_1)$. 
From Fact \ref{VD} again, $T_1 + \Con_{\tau_1} \vdash \varphi$. 
Therefore $\Th(T_0 + \Con_{\tau_0}) \subseteq \Th(T_1 + \Con_{\tau_1})$. 

Clause 2 follows from Clause 1. 
\end{proof}

The following corollary is an immediate consequence of Corollary \ref{VDCor1}.2. 

\begin{cor}\label{VDCor3}
If $\QPL_{\tau_0}(T_0) = \QPL_{\tau_1}(T_1)$ and $\Th(T_0) \subseteq \Th(T_1)$, then $T_1 \vdash \Con_{\tau_0} \leftrightarrow \Con_{\tau_1}$. 
\qed
\end{cor}

\section{On inclusions between quantified provability logics}

Inspired by Visser and de Jonge's fact, we explore further consequences of inclusion relationships between quantified provability logics that result from Artemov's Lemma. 

\subsection{Variations of Fact \ref{VD} and its consequences}

In this subsection, we prove variations of Visser and de Jonge's Fact \ref{VD} and its consequences. 
The following proposition is a variation of Fact \ref{VD} with respect to $\Sigma_1$ sentences. 

\begin{prop}\label{Sigma1_4}
For any $\Sigma_1$ sentence $\varphi$, the following are equivalent: 
\begin{enumerate}
	\item $T \vdash \varphi$. 
	\item $\chi^\circ \to \varphi^\circ \in \QPL_{\tau}(T)$. 
\end{enumerate}
\end{prop}
\begin{proof}
$(1 \Rightarrow 2)$: Suppose $T \vdash \varphi$. 
By Fact \ref{BooSigma_1}, for any arithmetical interpretation $f$, $\IS \vdash f_{\tau}(\chi^\circ) \land \varphi \to f_{\tau}(\varphi^\circ)$. 
Hence $T \vdash f_{\tau}(\chi^\circ \to \varphi^\circ)$. 
We have $\chi^\circ \to \varphi^\circ \in \QPL_{\tau}(T)$. 

$(2 \Rightarrow 1)$: This is trivial by considering a natural arithmetical interpretation. 
\end{proof}

Then we obtain a variation of Corollary \ref{VDCor1} by a similar proof. 

\begin{cor}\label{Sigma1_5}\leavevmode
\begin{enumerate}
	\item If $\QPL_{\tau_0}(T_0) \subseteq \QPL_{\tau_1}(T_1)$, then $\Th_{\Sigma_1}(T_0) \subseteq \Th_{\Sigma_1}(T_1)$; 
	\item If $\QPL_{\tau_0}(T_0) = \QPL_{\tau_1}(T_1)$, then $\Th_{\Sigma_1}(T_0) = \Th_{\Sigma_1}(T_1)$. 
\end{enumerate}
\qed
\end{cor}
%\begin{proof}
%1. Suppose $\QPL_{\tau_0}(T_0) \subseteq \QPL_{\tau_1}(T_1)$. 
%Let $\varphi$ be any $\Sigma_1$ sentence with $T_0 \vdash \varphi$. 
%Then by Proposition \ref{Sigma1_4}, $\chi^\circ \to \varphi^\circ \in \QPL_{\tau_0}(T_0)$. 
%By the supposition, $\chi^\circ \to \varphi^\circ \in \QPL_{\tau_1}(T_1)$. 
%By Proposition \ref{Sigma1_4} again, we obtain $T_1 \vdash \varphi$. 

%2 is immediate from clause 1. 
%\end{proof}

By applying Fact \ref{BooSigma_1}, Corollary \ref{VDCor3} is strengthened as follows. 

\begin{prop}\label{Basic1}
If $\QPL_{\tau_0}(T_0) \subseteq \QPL_{\tau_1}(T_1)$, then $T_1 \vdash \Con_{\tau_0} \leftrightarrow \Con_{\tau_1}$. 
\end{prop}
\begin{proof}
Suppose $\QPL_{\tau_0}(T_0) \subseteq \QPL_{\tau_1}(T_1)$. 
Then, $T_1 \vdash \Con_{\tau_1} \to \Con_{\tau_0}$ by Corollary \ref{VDCor1}.1, and so it suffices to prove $T_1 \vdash \Con_{\tau_0} \to \Con_{\tau_1}$. 
Let $f$ be any arithmetical interpretation. 
Since $\neg \Con_{\tau_0}$ is a $\Sigma_1$ sentence, by Fact \ref{BooSigma_1}, 
\[
	\IS \vdash f_{\tau_0}(\chi^\circ) \to (\neg \Con_{\tau_0} \to f_{\tau_0}(\neg \Con_{\tau_0}^\circ)). 
\]
Hence $T_0 \vdash f_{\tau_0}(\chi^\circ \land \Box \bot \to \neg \Con_{\tau_0}^\circ)$, and thus $\chi^\circ \land \Box \bot \to \neg \Con_{\tau_0}^\circ$ is in $\QPL_{\tau_0}(T_0)$. 
From the supposition, $\chi^\circ \land \Box \bot \to \neg \Con_{\tau_0}^\circ \in \QPL_{\tau_1}(T_1)$. 
By considering a natural arithmetical interpretation, we obtain that $T_1$ proves $\neg \Con_{\tau_1} \to \neg \Con_{\tau_0}$. 
Therefore $T_1 \vdash \Con_{\tau_0} \to \Con_{\tau_1}$. 
\end{proof}

\begin{cor}\label{PrCon}
If $T_1$ is consistent and $T_1 \vdash \Con_{\tau_0}$, then $\QPL_{\tau_0}(T_0) \nsubseteq \QPL_{\tau_1}(T_1)$. 
\end{cor}
\begin{proof}
Assume that $T_1$ is consistent and $T_1 \vdash \Con_{\tau_0}$. 
If $\QPL_{\tau_0}(T_0) \subseteq \QPL_{\tau_1}(T_1)$, then by Proposition \ref{Basic1}, $T_1 \vdash \Con_{\tau_0} \leftrightarrow \Con_{\tau_1}$. 
From the supposition, $T_1 \vdash \Con_{\tau_1}$ and this contradicts G\"odel's second incompleteness theorem. 
Therefore we get $\QPL_{\tau_0}(T_0) \nsubseteq \QPL_{\tau_1}(T_1)$. 
\end{proof}

The following corollary is a refinement of the result of Artemov \cite{Art86}. 

\begin{cor}
Suppose that $T$ is $\Sigma_1$-sound. 
Then, for any $\Sigma_1$ definition $\tau(v)$ of $T$, there exists a $\Sigma_1$ definition $\tau'(v)$ of $T$ such that $\QPL_{\tau}(T) \nsubseteq \QPL_{\tau'}(T)$ and $\QPL_{\tau'}(T) \nsubseteq \QPL_{\tau}(T)$. 
\end{cor}
\begin{proof}
Let $\tau(v)$ be any $\Sigma_1$ definition of $T$. 
Since $\neg \Con_\tau$ is $\Sigma_1$, by Fact \ref{DC}.3, $T \vdash \neg \Con_\tau \to \PR_{\tau}(\gdl{\neg \Con_\tau})$. 
Equivalently, $T \vdash \Con_{\tau + \Con_\tau} \to \Con_\tau$. 
Since $T$ is $\Sigma_1$-sound, $\Con_{\tau + \Con_\tau}$ is a true $\Pi_1$ sentence. 
Then, it is known that there exists a $\Sigma_1$ definition $\tau'(v)$ of $T$ such that $T \vdash \Con_{\tau'} \leftrightarrow \Con_{\tau + \Con_\tau}$ (cf.~Lindstr\"om \cite[Theorem 2.8.(b)]{Lin03}). 

Suppose, towards a contradiction, $T \vdash \Con_\tau \to \Con_{\tau'}$. 
Then, $T$ proves $\Con_\tau \to \Con_{\tau + \Con_\tau}$ and $\PR_{\tau}(\gdl{\neg \Con_\tau}) \to \neg \Con_\tau$. 
By L\"ob's theorem, $T$ also proves $\neg \Con_\tau$. 
This contradicts the $\Sigma_1$-soundness of $T$. 
Thus $T \nvdash \Con_\tau \to \Con_{\tau'}$. 

Moreover, $T \nvdash \Con_{\tau} \leftrightarrow \Con_{\tau'}$. 
It follows from Proposition \ref{Basic1} that $\QPL_{\tau}(T) \nsubseteq \QPL_{\tau'}(T)$ and $\QPL_{\tau'}(T) \nsubseteq \QPL_{\tau}(T)$. 
\end{proof}

\subsection{On provable equivalences of provability predicates}

In this subsection, we investigate further consequences of inclusions between quantified provability logics via Artemov's Lemma. 
In particular, we show that some provable equivalences of provability predicates are derived from inclusion. 
First, we prepare the following lemma. 

\begin{lem}\label{SelfProver1}
Let $f$ be any arithmetical interpretation. 
\begin{enumerate}
	\item $\PA \vdash f_\tau(\D) \to (R_f(x, y) \to \PR_{\tau}(\gdl{R_f(\dot{x}, \dot{y})}))$; 
	\item If $f(P_Z(x))$ and $f(P_S(x, y))$ are $\Sigma_1$ formulas, then $\IS \vdash R_f(x, y) \to \PR_{\tau}(\gdl{R_f(\dot{x}, \dot{y})})$. 
\end{enumerate}
\end{lem}
\begin{proof}
1. By the definition of $\D$, $f_\tau(P_Z(x)) \to \PR_\tau(\gdl{f_\tau(P_Z(\dot{x}))})$ and $f_\tau(P_S(x, y)) \to \PR_\tau(\gdl{f_\tau(P_S(\dot{x}, \dot{y}))})$ are provable in $\PA + f_\tau(\D)$. 
Also if $\PA \vdash \varphi_0 \to \PR_\tau(\gdl{\varphi_0})$ and $\PA \vdash \varphi_1 \to \PR_\tau(\gdl{\varphi_1})$, then $\PA \vdash \varphi_0 \land \varphi_1 \to \PR_{\tau}(\gdl{\varphi_0 \land \varphi_1})$ and $T \vdash \exists s \varphi_0 \to \PR_\tau(\gdl{\exists s \varphi_0})$. 
Thus it suffices to show that $\PA + f_\tau(\D)$ proves
\[
	 \forall z < x \ f_\tau(P_S((s)_z, (s)_{z+1})) \to \PR_\tau(\gdl{\forall z < \dot{x} \ f_\tau(P_S((\dot{s})_z, (\dot{s})_{z+1}))}). 
\]
Let $\psi(x)$ denote this formula. 
Since $T \vdash \forall z < 0 \ f_\tau(P_S((s)_z, (s)_{z+1}))$, by Fact \ref{DC}.1, $\PA \vdash \PR_\tau(\gdl{\forall z < 0 \ f_\tau(P_S((\dot{s})_z, (\dot{s})_{z+1}))})$. 
Thus $\PA \vdash \psi(0)$. 
Also $\PA + f_\tau(\D)$ proves
\begin{align*}
	 \psi(x) & \land \forall z < S(x) \ f_\tau(P_S((s)_z, (s)_{z+1}))\\
	& \qquad \to \forall z < x \ f_\tau(P_S((s)_z, (s)_{z+1})) \land f_\tau(P_S((s)_x, (s)_{x+1})),\\
	& \qquad \to \PR_\tau(\gdl{\forall z < \dot{x}\  f_\tau(P_S((\dot{s})_z, (\dot{s})_{z+1})) \land f_\tau(P_S((\dot{s})_{\dot{x}}, (\dot{s})_{\dot{x}+1}))}),\\
	& \qquad \to \PR_\tau(\gdl{\forall z < S(\dot{x})\  f_\tau(P_S((\dot{s})_z, (\dot{s})_{z+1}))}).
\end{align*}
Hence $\PA + f_\tau(\D) \vdash \psi(x) \to \psi(S(x))$, and by the induction axiom, we conclude $\PA + f_\tau(\D) \vdash \forall x \psi(x)$. 

2. If $f(P_Z(x))$ and $f(P_S(x, y))$ are $\Sigma_1$ formulas, then $R_f(x, y)$ is also a $\Sigma_1$ formula. 
Then the statement follows from Fact \ref{DC}.3. 
\end{proof}

We are ready to prove one of our main theorem of this subsection. 

\begin{thm}\label{IPPL1}
Suppose $\Th(\PA) \subseteq \Th(T_0)$. 
If $\QPL_{\tau_0}(T_0) \subseteq \QPL_{\tau_1}(T_1)$, then for any $\mathcal{L}_A$-formula $\varphi(\vec{y})$, 
\[
	T_1 \vdash \forall \vec{y} \left(\PR_{\tau_0}(\gdl{\Con_{\tau_0} \to \varphi(\vec{\dot{y}})}) \leftrightarrow \PR_{\tau_1}(\gdl{\Con_{\tau_1} \to \varphi(\vec{\dot{y}})}) \right). 
\]
\end{thm}
\begin{proof}
Suppose $\Th(\PA) \subseteq \Th(T_0)$ and $\QPL_{\tau_0}(T_0) \subseteq \QPL_{\tau_1}(T_1)$. 
Let $f$ be any arithmetical interpretation. 
By Artemov's Lemma, 
\[
	\IS \vdash \Con_{\tau_0} \land f_{\tau_0}(\D) \land f_{\tau_0}(\chi^\circ) \land R_f(\vec{x}, \vec{y}) \to \left(\varphi(\vec{x}) \leftrightarrow f_{\tau_0}(\varphi^\circ(\vec{y})) \right). 
\]
Then $T_0$ proves
\[
	f_{\tau_0}(\D) \land f_{\tau_0}(\chi^\circ) \land R_f(\vec{x}, \vec{y}) \to \left((\Con_{\tau_0} \to \varphi(\vec{x})) \leftrightarrow (\Con_{\tau_0} \to f_{\tau_0}(\varphi^\circ(\vec{y}))) \right). 
\]
By Fact \ref{DC}, we have 
\begin{align}\label{eq1}
	\IS & \vdash f_{\tau_0}(\Box \D) \land f_{\tau_0}(\Box \chi^\circ) \land \PR_{\tau_0}(\gdl{R_f(\vec{\dot{x}}, \vec{\dot{y}})}) \notag \\
	& \quad \to \left( \PR_{\tau_0}(\gdl{\Con_{\tau_0} \to \varphi(\vec{\dot{x}})}) \leftrightarrow f_{\tau_0}(\Box (\Diamond \top \to \varphi^\circ(\vec{y}))) \right). 
\end{align}
By Artemov's Lemma again, 
\begin{align}\label{eq2}
	\IS & \vdash \Con_{\tau_0} \land f_{\tau_0}(\D) \land f_{\tau_0}(\chi^\circ) \land R_f(\vec{x}, \vec{y}) \notag\\
	 & \quad \to \left(\PR_{\tau_0}(\gdl{\Con_{\tau_0} \to \varphi(\vec{\dot{x}})}) \leftrightarrow f_{\tau_0}(\PR_{\tau_0}(\gdl{\Con_{\tau_0} \to \varphi(\vec{\dot{y}})})^\circ) \right). 
\end{align}
From Lemma \ref{SelfProver1}.1, $\PA + f_{\tau_0}(\D) \vdash R_f(\vec{x}, \vec{y}) \to \PR_{\tau_0}(\gdl{R_f(\vec{\dot{x}}, \vec{\dot{y}})})$. 
By combining this with (\ref{eq1}) and (\ref{eq2}), we obtain
\begin{align*}
	\PA & \vdash \Con_{\tau_0} \land f_{\tau_0}(\boxdot \D) \land f_{\tau_0}(\boxdot \chi^\circ) \land R_f(\vec{x}, \vec{y})\\
	& \quad \to \left(f_{\tau_0}(\PR_{\tau_0}(\gdl{\Con_{\tau_0} \to \varphi(\vec{\dot{y}})})^\circ) \leftrightarrow f_{\tau_0}(\Box (\Diamond \top \to \varphi^\circ(\vec{y})))\right). 
\end{align*}
Since $\vec{x}$ does not appear in the consequent of the formula, 
\begin{align*}
	\PA & \vdash \Con_{\tau_0} \land f_{\tau_0}(\boxdot \D) \land f_{\tau_0}(\boxdot \chi^\circ) \land \exists \vec{x} R_f(\vec{x}, \vec{y})\\
	& \quad \to \left(f_{\tau_0}(\PR_{\tau_0}(\gdl{\Con_{\tau_0} \to \varphi(\vec{\dot{y}})})^\circ) \leftrightarrow f_{\tau_0}(\Box (\Diamond \top \to \varphi^\circ(\vec{y})))\right). 
\end{align*}
From Fact \ref{BooSurj}, $\IS \vdash \Con_{\tau_0} \land f_{\tau_0}(\D) \land f_{\tau_0}(\chi^\circ) \to \forall \vec{y} \exists \vec{x} R_f(\vec{x}, \vec{y})$. 
Hence
\begin{align*}
	\PA & \vdash \Con_{\tau_0} \land f_{\tau_0}(\boxdot \D) \land f_{\tau_0}(\boxdot \chi^\circ) \\
	& \quad \to \left(f_{\tau_0}(\PR_{\tau_0}(\gdl{\Con_{\tau_0} \to \varphi(\vec{\dot{y}})})^\circ) \leftrightarrow f_{\tau_0}(\Box (\Diamond \top \to \varphi^\circ(\vec{y})))\right). 
\end{align*}
Since $\Th(\PA) \subseteq \Th(T_0)$, we obtain that the sentence
\[
	\forall \vec{y} \left(\Diamond \top \land \boxdot \D \land \boxdot \chi^\circ \to \left(\PR_{\tau_0}(\gdl{\Con_{\tau_0} \to \varphi(\vec{\dot{y}})})^\circ \leftrightarrow \Box (\Diamond \top \to \varphi^\circ(\vec{y}))\right) \right)
\]
is contained in $\QPL_{\tau_0}(T_0)$. 
By the supposition, this sentence is also in $\QPL_{\tau_1}(T_1)$.
By considering a natural arithmetical interpretation and by Proposition \ref{Natural}, 
\[
	T_1 + \Con_{\tau_1} \vdash \forall \vec{y} \left( \PR_{\tau_0}(\gdl{\Con_{\tau_0} \to \varphi(\vec{\dot{y}})}) \leftrightarrow \PR_{\tau_1}(\gdl{\Con_{\tau_1} \to \varphi(\vec{\dot{y}})}) \right).
\]
By Proposition \ref{Basic1}, $T_1 \vdash \Con_{\tau_0} \to \Con_{\tau_1}$. 
Thus $T_1 + \neg \Con_{\tau_1} \vdash \neg \Con_{\tau_0}$, and hence
\[
	T_1 + \neg \Con_{\tau_1} \vdash \forall \vec{y} \left( \PR_{\tau_0}(\gdl{\Con_{\tau_0} \to \varphi(\vec{\dot{y}})}) \leftrightarrow \PR_{\tau_1}(\gdl{\Con_{\tau_1} \to \varphi(\vec{\dot{y}})}) \right). 
\]
Therefore we conclude
\[
	T_1 \vdash \forall \vec{y} \left(\PR_{\tau_0}(\gdl{\Con_{\tau_0} \to \varphi(\vec{\dot{y}})}) \leftrightarrow \PR_{\tau_1}(\gdl{\Con_{\tau_1} \to \varphi(\vec{\dot{y}})}) \right).
\]
\end{proof}

In our proof of Theorem \ref{IPPL1}, Lemma \ref{SelfProver1} is used to replace the formula $\PR_{\tau_0}(\ulcorner R_f(\vec{\dot{x}}, \vec{\dot{y}}) \urcorner)$ with $R_f(\vec{x}, \vec{y})$ in the antecedent of a formula. 
If $\varphi$ is a sentence, then this procedure is no longer needed, and so the proof proceeds without using Lemma \ref{SelfProver1}. 
Then other parts of our proof of Theorem \ref{IPPL1} work within $\IS$. 
Thus we also obtain the following theorem. 

\begin{thm}\label{IPPL2}
If $\QPL_{\tau_0}(T_0) \subseteq \QPL_{\tau_1}(T_1)$, then for any $\mathcal{L}_A$-sentence $\varphi$, 
\[
	T_1 \vdash \PR_{\tau_0}(\gdl{\Con_{\tau_0} \to \varphi}) \leftrightarrow \PR_{\tau_1}(\gdl{\Con_{\tau_1} \to \varphi}). 
\]
\qed
\end{thm}

Using Fact \ref{BooSigma_1}, we prove a variation of Theorem \ref{IPPL1} with respect to $\Pi_1$ formulas. 

\begin{thm}\label{Pi1_1}
Suppose $\Th(\PA) \subseteq \Th(T_0)$. 
If $\QPL_{\tau_0}(T_0) \subseteq \QPL_{\tau_1}(T_1)$, then for any $\Pi_1$ formula $\varphi(\vec{y})$, 
\[
	T_1 \vdash \forall \vec{y} (\PR_{\tau_1}(\gdl{\varphi(\vec{\dot{y}})}) \to \PR_{\tau_0}(\gdl{\varphi(\vec{\dot{y}})})). 
\]
\end{thm}
\begin{proof}
Suppose $\Th(\PA) \subseteq \Th(T_0)$ and $\QPL_{\tau_0}(T_0) \subseteq \QPL_{\tau_1}(T_1)$. 
Let $f$ be any arithmetical interpretation and let $\varphi(\vec{y})$ be any $\Pi_1$ formula. 
Since $\neg \varphi(\vec{y})$ is $\Sigma_1$, by Fact \ref{BooSigma_1}, $\IS \vdash f_{\tau_0}(\chi^\circ) \land R_f(\vec{x}, \vec{y}) \land \neg \varphi(\vec{x}) \to f_{\tau_0}(\neg \varphi^\circ(\vec{y}))$. 
Then, $T_0 \vdash f_{\tau_0}(\chi^\circ) \land R_f(\vec{x}, \vec{y}) \land f_{\tau_0}(\varphi^\circ(\vec{y})) \to \varphi(\vec{x})$. 
By Fact \ref{DC}, 
\begin{equation}\label{eq5}
	\IS \vdash f_{\tau_0}(\Box \chi^\circ) \land \PR_{\tau_0}(\gdl{R_f(\vec{\dot{x}}, \vec{\dot{y}})}) \land f_{\tau_0}(\Box \varphi^\circ(\vec{y})) \to \PR_{\tau_0}(\gdl{\varphi(\vec{\dot{x}})}).
\end{equation}
By Artemov's Lemma, $\IS$ proves 
\begin{equation}\label{eq6}
	\Con_{\tau_0} \land f_{\tau_0}(\D) \land f_{\tau_0}(\chi^\circ) \land R_f(\vec{x}, \vec{y}) \land \PR_{\tau_0}(\gdl{\varphi(\vec{\dot{x}})}) \to f_{\tau_0}(\PR_{\tau_0}(\gdl{\varphi(\vec{\dot{y}})})^\circ). 
\end{equation}
By combining Lemma \ref{SelfProver1} with (\ref{eq5}) and (\ref{eq6}), $\PA$ proves 
\[
	\Con_{\tau_0} \land f_{\tau_0}(\D) \land f_{\tau_0}(\boxdot \chi^\circ) \land R_f(\vec{x}, \vec{y}) \land f_{\tau_0}(\Box \varphi^\circ(\vec{y})) \to f_{\tau_0}(\PR_{\tau_0}(\gdl{\varphi(\vec{\dot{y}})})^\circ). 
\]
As in the proof of Theorem \ref{IPPL1}, $R_f(\vec{x}, \vec{y})$ is removed from the antecedent of the formula, that is, 
\[
	\PA \vdash \Con_{\tau_0} \land f_{\tau_0}(\D) \land f_{\tau_0}(\boxdot \chi^\circ) \land f_{\tau_0}(\Box \varphi^\circ(\vec{y})) \to f_{\tau_0}(\PR_{\tau_0}(\gdl{\varphi(\vec{\dot{y}})})^\circ). 
\]
Since $\Th(\PA) \subseteq \Th(T_0)$, 
\[
	\forall \vec{y} \left(\Diamond \top \land \D \land \boxdot \chi^\circ \land \Box \varphi^\circ(\vec{y}) \to \PR_{\tau_0}(\gdl{\varphi(\vec{\dot{y}})})^\circ \right) \in \QPL_{\tau_0}(T_0) \subseteq \QPL_{\tau_1}(T_1). 
\]
By considering a natural arithmetical interpretation, we obtain
\[
	T_1 + \Con_{\tau_1} \vdash \forall \vec{y} ( \PR_{\tau_1}(\gdl{\varphi(\vec{\dot{y}})}) \to \PR_{\tau_0}(\gdl{\varphi(\vec{\dot{y}})})). 
\]
By Proposition \ref{Basic1}, $T_1 + \neg \Con_{\tau_1} \vdash \neg \Con_{\tau_0}$, and in particular, $T_1 + \neg \Con_{\tau_1}$ proves $\forall \vec{y} \PR_{\tau_0}(\gdl{\varphi(\vec{\dot{y}})})$. 
Therefore we conclude
\[
	T_1 \vdash \forall \vec{y} (\PR_{\tau_1}(\gdl{\varphi(\vec{\dot{y}})}) \to \PR_{\tau_0}(\gdl{\varphi(\vec{\dot{y}})})). 
\]
\end{proof}

As above, we also obtain the following theorem. 

\begin{thm}\label{Pi1_2}
If $\QPL_{\tau_0}(T_0) \subseteq \QPL_{\tau_1}(T_1)$, then for any $\Pi_1$ sentence $\varphi$, 
\[
	T_1 \vdash \PR_{\tau_1}(\gdl{\varphi}) \to \PR_{\tau_0}(\gdl{\varphi}). 
\]
\qed
\end{thm}

As consequences of theorems proved in this subsection, we obtain several corollaries. 

\begin{cor}\label{DeducEquiv}
If $\QPL_{\tau_0}(T_0) \subseteq \QPL_{\tau_1}(T_1)$ and $T_1$ is $\Sigma_1$-sound, then 
\begin{enumerate}
	\item $\Th(T_0 + \Con_{\tau_0}) = \Th(T_1 + \Con_{\tau_1})$; and
	\item $\Th_{\Pi_1}(T_1) \subseteq \Th_{\Pi_1}(T_0)$. 
\end{enumerate}
\end{cor}
\begin{proof}
Suppose $\QPL_{\tau_0}(T_0) \subseteq \QPL_{\tau_1}(T_1)$ and $T_1$ is $\Sigma_1$-sound. 

1. By Corollary \ref{VDCor1}.1, $\Th(T_0 + \Con_{\tau_0}) \subseteq \Th(T_1 + \Con_{\tau_1})$. 
On the other hand, let $\varphi$ be any $\mathcal{L}_A$-sentence $\varphi$ with $T_1 + \Con_{\tau_1} \vdash \varphi$. 
Then, $T_1 \vdash \PR_{\tau_1}(\gdl{\Con_{\tau_1} \to \varphi})$ by Fact \ref{DC}.1. 
By Theorem \ref{IPPL2}, 
\[
	T_1 \vdash \PR_{\tau_0}(\gdl{\Con_{\tau_0} \to \varphi}) \leftrightarrow \PR_{\tau_1}(\gdl{\Con_{\tau_1} \to \varphi}), 
\]
and hence $T_1 \vdash \PR_{\tau_0}(\gdl{\Con_{\tau_0} \to \varphi})$. 
Then, $\PR_{\tau_0}(\gdl{\Con_{\tau_0} \to \varphi})$ is true in $\N$ because $T_1$ is $\Sigma_1$-sound. 
This means $T_0 \vdash \Con_{\tau_0} \to \varphi$. 
Therefore we conclude $\Th(T_1 + \Con_{\tau_1}) \subseteq \Th(T_0 + \Con_{\tau_0})$. 

2. Let $\varphi$ be any $\Pi_1$ sentence such that $T_1 \vdash \varphi$. 
Then $T_1 \vdash \PR_{\tau_1}(\gdl{\varphi})$ by Fact \ref{DC}.1. 
By Theorem \ref{Pi1_2}, $T_1 \vdash \PR_{\tau_1}(\gdl{\varphi}) \to \PR_{\tau_0}(\gdl{\varphi})$, and hence $T_1 \vdash \PR_{\tau_0}(\gdl{\varphi})$. 
Since $T_1$ is $\Sigma_1$-sound, $T_0 \vdash \varphi$. 
Thus $\Th_{\Pi_1}(T_1) \subseteq \Th_{\Pi_1}(T_0)$. 

\end{proof}

In the next subsection, we will prove that the assumption of the $\Sigma_1$-soundness of $T_1$ in the statement of Corollary \ref{DeducEquiv} cannot be removed (see Propositions \ref{Sigma1_3} and \ref{notPicons}). 

\begin{rem}
We say that a theory $T_1$ is \textit{faithfully interpretable} in a theory $T_0$ if there exists an interpretation $I$ of $T_1$ in $T_0$ such that for any $\mathcal{L}_A$-sentence $\varphi$, $T_1 \vdash \varphi$ if and only if $T_0 \vdash I(\varphi)$. 
Lindstr\"om \cite{Lin84} proved that if $T_0$ and $T_1$ are consistent recursively enumerable extensions of $\PA$, then $T_1$ is faithfully interpretable in $T_0$ if and only if $\Th_{\Pi_1}(T_1) \subseteq \Th_{\Pi_1}(T_0)$ and $\Th_{\Sigma_1}(T_0) \subseteq \Th_{\Sigma_1}(T_1)$. 
Therefore from Corollaries \ref{DeducEquiv} and \ref{Sigma1_5}.1, we obtain that if $\Th(\PA) \subseteq \Th(T_0) \cap \Th(T_1)$, $\QPL_{\tau_0}(T_0) \subseteq \QPL_{\tau_1}(T_1)$ and $T_1$ is $\Sigma_1$-sound, then $T_1$ is faithfully interpretable in $T_0$.
\end{rem}

We show that if $T_1$ is $\Sigma_1$-sound and proves the $\Sigma_1$-soundness of $T_0$, then $\QPL_{\tau_0}(T_0)$ and $\QPL_{\tau_1}(T_1)$ are incomparable in the following strong sense. 

\begin{cor}\label{Sigma_1Rfn}
Suppose that $T_0$ is consistent, $T_1$ is $\Sigma_1$-sound and for some $\Sigma_1$ definition $\sigma_0(v)$ of $T_0$, for all $\Sigma_1$ sentences $\varphi$, $T_1 \vdash \PR_{\sigma_0}(\gdl{\varphi}) \to \varphi$. 
Then, for any respective $\Sigma_1$ definitions $\tau_0(v)$ and $\tau_1(v)$ of $T_0$ and $T_1$, $\QPL_{\tau_0}(T_0) \nsubseteq \QPL_{\tau_1}(T_1)$ and $\QPL_{\tau_1}(T_1) \nsubseteq \QPL_{\tau_0}(T_0)$. 
\end{cor}
\begin{proof}
First, we show $\QPL_{\tau_0}(T_0) \nsubseteq \QPL_{\tau_1}(T_1)$. 
By the supposition, $T_1$ proves $\PR_{\sigma_0}(\gdl{0=\overline{1}}) \to 0 = \overline{1}$ which is equivalent to $\Con_{\sigma_0}$. 
On the other hand, $T_0 \nvdash \Con_{\sigma_0}$ by the second incompleteness theorem. 
Since $\Con_{\sigma_0}$ is a $\Pi_1$ sentence, $\Th_{\Pi_1}(T_1) \nsubseteq \Th_{\Pi_1}(T_0)$. 
Therefore $\QPL_{\tau_0}(T_0) \nsubseteq \QPL_{\tau_1}(T_1)$ by Corollary \ref{DeducEquiv} because $T_1$ is $\Sigma_1$-sound. 

Secondly, we show $\QPL_{\tau_1}(T_1) \nsubseteq \QPL_{\tau_0}(T_0)$. 
Since $\neg \Con_{\tau_0}$ is $\Sigma_1$, $T_1$ proves $\PR_{\sigma_0}(\gdl{\neg \Con_{\tau_0}}) \to \neg \Con_{\tau_0}$, and also proves $\Con_{\tau_0} \to \Con_{\sigma_0 + \Con_{\tau_0}}$. 
On the other hand, assume, towards a contradiction, that $T_0 + \Con_{\tau_0}$ proves the sentence $\Con_{\tau_0} \to \Con_{\sigma_0 + \Con_{\tau_0}}$. 
Then, $T_0 + \Con_{\tau_0}$ proves its own consistency, and hence it is inconsistent by the second incompleteness theorem. 
We have $T_0 \vdash \neg \Con_{\tau_0}$. 
By Fact \ref{DC}.1, $T_1 \vdash \PR_{\sigma_0}(\gdl{\neg \Con_{\tau_0}})$. 
Hence $T_1 \vdash \neg \Con_{\tau_0}$, and this contradicts the $\Sigma_1$-soundness of $T_1$. 
We obtain $T_0 + \Con_{\tau_0} \nvdash \Con_{\tau_0} \to \Con_{\sigma_0 + \Con_{\tau_0}}$. 
Therefore $\Th(T_1) \nsubseteq \Th(T_0 + \Con_{\tau_0})$. 
By Corollary \ref{VDCor1}.1, we conclude $\QPL_{\tau_1}(T_1) \nsubseteq \QPL_{\tau_0}(T_0)$.
\end{proof}

\begin{rem}
Let $i$ and $j$ be any natural numbers with $0 < i < j$. 
Then, the theory $\mathbf{I\Sigma_j}$ is $\Sigma_1$-sound and proves $\PR_{\mathbf{I\Sigma_i}}(\gdl{\varphi}) \to \varphi$ for all $\Sigma_1$ sentences $\varphi$ (cf.~H\'ajek and Pudl\'ak \cite[Corollary I.4.34]{HP}). 
From Corollary \ref{Sigma_1Rfn}, for any respective $\Sigma_1$ definitions $\sigma_i(v)$ and $\sigma_j(v)$ of $\mathbf{I\Sigma_i}$ and $\mathbf{I\Sigma_j}$, $\QPL_{\sigma_i}(\mathbf{I\Sigma_i}) \nsubseteq \QPL_{\sigma_j}(\mathbf{I\Sigma_j})$ and $\QPL_{\sigma_j}(\mathbf{I\Sigma_j}) \nsubseteq \QPL_{\sigma_i}(\mathbf{I\Sigma_i})$. 
This is a refinement of a result of Kurahashi \cite{Kur13B}. 
\end{rem}

\begin{lem}\label{eqQPL}
Let $\sigma(v)$ be any $\Sigma_1$ definition of some theory. 
Suppose that for all $\mathcal{L}_A$-formulas $\varphi(\vec{x})$, $T \vdash \forall \vec{x} (\PR_{\sigma}(\ulcorner \varphi(\vec{\dot{x}}) \urcorner) \leftrightarrow \PR_{\tau}(\ulcorner \varphi(\vec{\dot{x}}) \urcorner))$. 
Then, for any quantified modal formula $A$ and any arithmetical interpretation $f$, $T \vdash f_{\sigma}(A) \leftrightarrow f_{\tau}(A)$.
\end{lem}
\begin{proof}
We prove the lemma by induction on the construction of $A$. 
We only give a proof of the case that $A$ is of the form $\Box B$. 
Assume that $T$ proves $f_{\sigma}(B) \leftrightarrow f_{\tau}(B)$. 
Then, by Fact \ref{DC}, $\IS \vdash \PR_{\tau}(\gdl{f_{\sigma}(B)}) \leftrightarrow f_{\tau}(\Box B)$. 
Since $T \vdash f_{\sigma}(\Box B) \leftrightarrow \PR_{\tau}(\gdl{f_{\sigma}(B)})$ by the supposition, we obtain that $f_{\sigma}(\Box B) \leftrightarrow f_{\tau}(\Box B)$ is provable in $T$. 
\end{proof}

\begin{cor}\label{IPPL3}
If $\Th(\PA) \subseteq \Th(T_0)$ and $\QPL_{\tau_0}(T_0) \subseteq \QPL_{\tau_1}(T_1)$, then $\QPL_{\tau_0 + \Con_{\tau_0}}(T_0 + \Con_{\tau_0}) \subseteq \QPL_{\tau_1 + \Con_{\tau_1}}(T_1 + \Con_{\tau_1})$.
\end{cor}
\begin{proof}
Suppose $\Th(\PA) \subseteq \Th(T_0)$ and $\QPL_{\tau_0}(T_0) \subseteq \QPL_{\tau_1}(T_1)$. 
Let $A$ be any element of $\QPL_{\tau_0 + \Con_{\tau_0}}(T_0 + \Con_{\tau_0})$ and $f$ be an arbitrary arithmetical interpretation. 
Then, $T_0 + \Con_{\tau_0} \vdash f_{\tau_0 + \Con_{\tau_0}}(A)$. 
Since $\Th(T_0 + \Con_{\tau_0}) \subseteq \Th(T_1 + \Con_{\tau_1})$ by Corollary \ref{VDCor1}.1, $T_1 + \Con_{\tau_1} \vdash f_{\tau_0 + \Con_{\tau_0}}(A)$. 
By Theorem \ref{IPPL1}, for any $\mathcal{L}_A$-formula $\varphi(\vec{x})$, 
\[
	T_1 \vdash \forall \vec{x} \left(\PR_{\tau_0 + \Con_{\tau_0}}(\gdl{\varphi(\vec{\dot{x}})}) \leftrightarrow \PR_{\tau_1 + \Con_{\tau_1}}(\gdl{\varphi(\vec{\dot{x}})}) \right). 
\]
Thus by Lemma \ref{eqQPL}, $T_1 + \Con_{\tau_1} \vdash f_{\tau_0 + \Con_{\tau_0}}(A) \leftrightarrow f_{\tau_1 + \Con_{\tau_1}}(A)$, and hence $T_1 + \Con_{\tau_1} \vdash f_{\tau_1 + \Con_{\tau_1}}(A)$. 
Since $f$ is arbitrary, $A$ is contained in $\QPL_{\tau_1 + \Con_{\tau_1}}(T_1 + \Con_{\tau_1})$. 
\end{proof}

Moreover, we strengthen Proposition \ref{Basic1} and Corollary \ref{IPPL3}. 

\begin{defn}
We define a sequence $(\Con_\tau^n)_{n \in \N}$ of $\Pi_1$ consistency statements of $T$ inductively as follows: 
\begin{enumerate}
	\item $\Con_\tau^0 :\equiv 0 = 0$; and
	\item $\Con_\tau^{n+1} :\equiv \Con_{\tau + \Con_{\tau}^n}$. 
\end{enumerate}
\end{defn}

Since $\neg \Con_{\tau}^n$ is a $\Sigma_1$ sentence, $\IS \vdash \neg \Con_{\tau}^n \to \PR_{\tau}(\gdl{\neg \Con_\tau^n})$ by Fact \ref{DC}.3. 
Equivalently, $\IS \vdash \Con_\tau^{n+1} \to \Con_\tau^n$. 
Thus $\Con_\tau^n \land \Con_{\tau + \Con_\tau^n}$ is provably equivalent to $\Con_\tau^{n+1}$ over $\IS$.  

\begin{cor}\label{IPPL4}
If $\Th(\PA) \subseteq \Th(T_0)$ and $\QPL_{\tau_0}(T_0) \subseteq \QPL_{\tau_1}(T_1)$, then for any natural number $n \geq 1$, 
\begin{enumerate}
	\item $\QPL_{\tau_0 + \Con_{\tau_0}^n}(T_0 + \Con_{\tau_0}^n) \subseteq \QPL_{\tau_1 + \Con_{\tau_1}^n}(T_1 + \Con_{\tau_1}^n)$; and
	\item $T_1 \vdash \Con_{\tau_0}^n \leftrightarrow \Con_{\tau_1}^n$. 
\end{enumerate}
\end{cor}
\begin{proof}
Suppose $\Th(\PA) \subseteq \Th(T_0)$ and $\QPL_{\tau_0}(T_0) \subseteq \QPL_{\tau_1}(T_1)$. 

1. By induction on $n \geq 1$. 
For $n=1$, the statement is exactly Corollary \ref{IPPL3}. 
Suppose $\QPL_{\tau_0 + \Con_{\tau_0}^n}(T_0 + \Con_{\tau_0}^n) \subseteq \QPL_{\tau_1 + \Con_{\tau_1}^n}(T_1 + \Con_{\tau_1}^n)$. 
As commented above, $\Con_{\tau_i}^n \land \Con_{\tau_i + \Con_{\tau_i}^n}$ is equivalent to $\Con_{\tau_i}^{n+1}$ for $i \in \{0, 1\}$, and hence by Corollary \ref{IPPL3}, 
\[
	\QPL_{\tau_0 + \Con_{\tau_0}^{n+1}}(T_0 + \Con_{\tau_0}^{n+1}) \subseteq \QPL_{\tau_1 + \Con_{\tau_1}^{n+1}}(T_1 + \Con_{\tau_1}^{n+1}). 
\]

2. By induction on $n \geq 1$. 
For $n=1$, the statement is exactly Proposition \ref{Basic1}. 
Suppose $T_1 \vdash \Con_{\tau_0}^n \leftrightarrow \Con_{\tau_1}^n$. 
By Clause 1, 
\[
	\QPL_{\tau_0 + \Con_{\tau_0}^n}(T_0 + \Con_{\tau_0}^n) \subseteq \QPL_{\tau_1 + \Con_{\tau_1}^n}(T_1 + \Con_{\tau_1}^n).
\] 
Then by Proposition \ref{Basic1}, $T_1 + \Con_{\tau_1}^n$ proves $\Con_{\tau_0 + \Con_{\tau_0}^n} \leftrightarrow \Con_{\tau_1 + \Con_{\tau_1}^n}$. 
This means 
\begin{equation}\label{eqX}
	T_1 + \Con_{\tau_1}^n \vdash \Con_{\tau_0}^{n+1} \leftrightarrow \Con_{\tau_1}^{n+1}.
\end{equation}
We prove $T_1 \vdash \Con_{\tau_0}^{n+1} \leftrightarrow \Con_{\tau_1}^{n+1}$. 
Since $T_1 + \Con_{\tau_1}^{n+1} \vdash \Con_{\tau_1}^n$, it follows from (\ref{eqX}) that $T_1 + \Con_{\tau_1}^{n+1} \vdash \Con_{\tau_0}^{n+1}$. 
Conversely, since $T_1 + \Con_{\tau_0}^{n+1} \vdash \Con_{\tau_0}^n$, $T_1 + \Con_{\tau_0}^{n+1} \vdash \Con_{\tau_1}^n$ by induction hypothesis. 
Then, $T_1 + \Con_{\tau_0}^{n+1} \vdash \Con_{\tau_1}^{n+1}$ from (\ref{eqX}). 
\end{proof}

Under certain suppositions, we give the following necessary and sufficient condition for $\QPL_{\tau_0}(T_0) \subseteq \QPL_{\tau_1}(T_1)$. 

\begin{cor}\label{QPL_characterization}
Suppose that $\Th(T_0) \subseteq \Th(T_1)$ and there exists a $\Pi_1$ sentence $\pi$ satisfying the following two conditions: 
\begin{itemize}
	\item $T_0 \vdash \Con_{\tau_0} \to \neg \PR_{\tau_0}(\gdl{\pi})$; 
	\item $T_1 \vdash \PR_{\tau_1}(\gdl{\pi})$. 
\end{itemize}
Then, $\QPL_{\tau_0}(T_0) \subseteq \QPL_{\tau_1}(T_1)$ if and only if $T_1 \vdash \neg \Con_{\tau_0} \land \neg \Con_{\tau_1}$. 
\end{cor}
\begin{proof}
$(\Rightarrow)$: Suppose $\QPL_{\tau_0}(T_0) \subseteq \QPL_{\tau_1}(T_1)$. 
Let $\pi$ be a $\Pi_1$ sentence satisfying the two conditions stated above. 
By Theorem \ref{Pi1_2}, $\PR_{\tau_1}(\gdl{\pi}) \to \PR_{\tau_0}(\gdl{\pi})$ is provable in $T_1$, and hence $T_1 \vdash \PR_{\tau_0}(\gdl{\pi})$ by the choice of $\pi$. 
On the other hand, by Corollary \ref{VDCor1}.1, $\Th(T_0 + \Con_{\tau_0}) \subseteq \Th(T_1 + \Con_{\tau_1})$, and thus $T_1 + \Con_{\tau_1} \vdash \neg \PR_{\tau_0}(\gdl{\pi})$. 
Therefore $T_1 + \Con_{\tau_1}$ is inconsistent, and we obtain $T_1 \vdash \neg \Con_{\tau_1}$. 
By Proposition \ref{Basic1}, $T_1 \vdash \Con_{\tau_0} \to \Con_{\tau_1}$. 
Hence $T_1 \vdash \neg \Con_{\tau_0}$. 

$(\Leftarrow)$: Assume that $T_1$ proves $\neg \Con_{\tau_0}$ and $\neg \Con_{\tau_1}$. 
Then, for any $\mathcal{L}_A$-formula $\varphi(\vec{x})$, $T_1 \vdash \forall \vec{x}(\PR_{\tau_0}(\gdl{\varphi(\vec{\dot{x}})}) \leftrightarrow \PR_{\tau_1}(\gdl{\varphi(\vec{\dot{x}})}))$. 
Let $A$ be any element of $\QPL_{\tau_0}(T_0)$ and $f$ be any arithmetical interpretation. 
Then, $T_0 \vdash f_{\tau_0}(A)$. 
Since $\Th(T_0) \subseteq \Th(T_1)$, $T_1 \vdash f_{\tau_0}(A)$. 
By Lemma \ref{eqQPL}, $f_{\tau_0}(A) \leftrightarrow f_{\tau_1}(A)$ is provable in $T_1$, and hence $T_1 \vdash f_{\tau_1}(A)$. 
Therefore $A \in \QPL_{\tau_1}(T_1)$. 
We have proved $\QPL_{\tau_0}(T_0) \subseteq \QPL_{\tau_1}(T_1)$. 
\end{proof}

For example, for any $\Pi_1$ sentence $\pi$ satisfying $T_0 \vdash \Con_{\tau_0} \to \neg \PR_{\tau_0}(\gdl{\pi})$, the theories $T_0$ and $T_1 : = T_0 + \pi$ satisfy the assumption of Corollary \ref{QPL_characterization}. 
Corollary \ref{QPL_characterization} is used in the proof of Proposition \ref{Ex3} below.

\subsection{Some counterexamples}

In this subsection, we give some counterexamples to several statements. 
Before giving them, we prepare a lemma.

\begin{lem}\label{Sigma_1}
For any $\mathcal{L}_A$-sentence $\varphi$ with $T \vdash \varphi \to \PR_{\tau}(\gdl{\varphi})$, 
\[
	\QPL_{\tau}(T) \subseteq \QPL_{\tau + \varphi}(T + \varphi).
\]
\end{lem}
\begin{proof}
Suppose $T \vdash \varphi \to \PR_{\tau}(\gdl{\varphi})$. 
Let $A$ be any element of $\QPL_\tau(T)$ and $f$ be any arithmetical interpretation. 
Then, $T \vdash f_{\tau}(A)$. 
Since $T + \varphi$ proves $\PR_\tau(\gdl{\varphi})$, for any $\mathcal{L}_A$-formula $\psi(\vec{x})$, it follows from Fact \ref{DC}.2 that
\begin{align*}
	T + \varphi \vdash \PR_{\tau + \varphi}(\gdl{\psi(\vec{\dot{x}})}) & \leftrightarrow \PR_{\tau}(\gdl{\varphi \to \psi(\vec{\dot{x}})}), \\
	& \leftrightarrow \PR_{\tau}(\gdl{\psi(\vec{\dot{x}})}). 
\end{align*}
Then by Lemma \ref{eqQPL}, $T + \varphi \vdash f_{\tau}(A) \leftrightarrow f_{\tau + \varphi}(A)$. 
Hence $T + \varphi \vdash f_{\tau + \varphi}(A)$. 
We conclude $\QPL_{\tau}(T) \subseteq \QPL_{\tau + \varphi}(T + \varphi)$.
\end{proof}

The following two propositions show that in the statement of Corollary \ref{DeducEquiv}, the assumption of the $\Sigma_1$-soundness of $T_1$ cannot be omitted. 

\begin{prop}\label{Sigma1_3}
There exist consistent recursively enumerable extensions $T_0$ and $T_1$ of $\IS$ and respective $\Sigma_1$ definitions $\tau_0(v)$ and $\tau_1(v)$ of $T_0$ and $T_1$ satisfying the following conditions: 
\begin{enumerate}
	\item $\QPL_{\tau_0}(T_0) \subseteq \QPL_{\tau_1}(T_1)$; 
	\item $T_0 + \Con_{\tau_0}$ and $T_1 + \Con_{\tau_1}$ are consistent; and
	\item $\Th(T_1 + \Con_{\tau_1}) \nsubseteq \Th(T_0 + \Con_{\tau_0})$.
\end{enumerate}
\end{prop}
\begin{proof}
Let $T_0$ be any $\Sigma_1$-sound recursively enumerable extension of $\IS$ and $\tau_0(v)$ be any $\Sigma_1$ definition of $T_0$. 
Also let $\varphi$ be the $\Sigma_1$ sentence $\neg \Con_{\tau_0}^2$. 
Then $\N \models \neg \varphi$. 
Let $T_1 : = T_0 + \varphi$ and $\tau_1(v)$ be $(\tau_0 + \varphi)(v)$. 

1. Since $\varphi$ is a $\Sigma_1$ sentence, $T_0 \vdash \varphi \to \PR_{\tau_0}(\ulcorner \varphi \urcorner)$ by Fact \ref{DC}.3. 
Then by Lemma \ref{Sigma_1}, $\QPL_{\tau_0}(T_0) \subseteq \QPL_{\tau_1}(T_1)$. 

2. Since $T_0$ is $\Sigma_1$-sound, $T_0 + \Con_{\tau_0}$ is consistent. 
Suppose, towards a contradiction, that $T_1 + \Con_{\tau_1}$ is inconsistent. 
Then $T_0 + \varphi \vdash \neg \Con_{\tau_0 + \varphi}$, and hence $T_0 \vdash \varphi \to \PR_{\tau_0}(\gdl{\neg \varphi})$. 
Since $T_0 \vdash \varphi \to \PR_{\tau_0}(\gdl{\varphi})$, we have $T_0 \vdash \varphi \to \neg \Con_{\tau_0}$. 
It follows $T_0 \vdash \PR_{\tau_0}(\gdl{\neg \Con_{\tau_0}}) \to \neg \Con_{\tau_0}$. 
By L\"ob's theorem, $T_0 \vdash \neg \Con_{\tau_0}$. 
This contradicts the $\Sigma_1$-soundness of $T_0$. 
Therefore $T_1 + \Con_{\tau_1}$ is consistent. 

3. Since $T_0 + \Con_{\tau_0}$ is also $\Sigma_1$-sound, $T_0 + \Con_{\tau_0} \nvdash \varphi$. 
On the other hand, $T_1 + \Con_{\tau_1} \vdash \varphi$, and hence $\Th(T_1 + \Con_{\tau_1}) \nsubseteq \Th(T_0 + \Con_{\tau_0})$. 
\end{proof}

\begin{prop}\label{notPicons}
There exist consistent recursively enumerable extensions $T_0$ and $T_1$ of $\IS$ and respective $\Sigma_1$ definitions $\tau_0(v)$ and $\tau_1(v)$ of $T_0$ and $T_1$ satisfying the following conditions: 
\begin{enumerate}
	\item $\QPL_{\tau_0}(T_0) \subseteq \QPL_{\tau_1}(T_1)$; and
	\item $\Th_{\Pi_1}(T_1) \nsubseteq \Th_{\Pi_1}(T_0)$.
\end{enumerate}
\end{prop}
\begin{proof}
Let $T_0$ be an arbitrary consistent recursively enumerable extension of $\IS$ and $\tau_0(v)$ be any $\Sigma_1$ definition of $T_0$. 
Let $\rho$ be a $\Pi_1$ Rosser sentence of $T_0$ defined by using $\tau_0(v)$, and let $T_1 : = T_0 + \neg \rho$ and $\tau_1(v)$ be $(\tau_0 + \neg \rho)(v)$. 
By Rosser's theorem, $T_1$ is consistent. 
Since $\neg \rho$ is $\Sigma_1$, by Lemma \ref{Sigma_1}, $\QPL_{\tau_0}(T_0) \subseteq \QPL_{\tau_1}(T_1)$. 
It is easily shown that there exists a $\Pi_1$ sentence $\pi$ such that $\IS \vdash \rho \lor \pi$ and $\IS \vdash \rho \land \pi \to \Con_{\tau_0}$. 
Then $T_1 \vdash \pi$ and $T_0 \nvdash \pi$ because $T_0 \nvdash \rho \to \Con_{\tau_0}$. 
Therefore $\Th_{\Pi_1}(T_1) \nsubseteq \Th_{\Pi_1}(T_0)$ (see also Lindstr\"om \cite[Chapter 5 Exercise 1]{Lin03}). 
\end{proof}

The following proposition shows that the converse implications of Proposition \ref{Basic1}, Theorem \ref{IPPL1} and Corollary \ref{DeducEquiv} do not hold. 

\begin{prop}\label{Ex3}
There exist consistent recursively enumerable extensions $T_0$ and $T_1$ of $\IS$ and respective $\Sigma_1$ definitions $\tau_0(v)$ and $\tau_1(v)$ of $T_0$ and $T_1$ satisfying the following conditions:
\begin{enumerate}
	\item $\IS \vdash \Con_{\tau_0} \leftrightarrow \Con_{\tau_1}$; 
	\item $T_1$ is $\Sigma_1$-sound and $\Th(T_0 + \Con_{\tau_0}) = \Th(T_1 + \Con_{\tau_1})$; 
	\item For any $\mathcal{L}_A$-formula $\varphi(\vec{x})$,
\[
	\IS \vdash \forall \vec{x} \left(\PR_{\tau_0}(\gdl{\Con_{\tau_0} \to \varphi(\vec{\dot{x}})}) \leftrightarrow \PR_{\tau_1}(\gdl{\Con_{\tau_1} \to \varphi(\vec{\dot{x}})})\right); 
\]
	\item $\QPL_{\tau_0}(T_0) \nsubseteq \QPL_{\tau_1}(T_1)$. 
\end{enumerate}
\end{prop}
\begin{proof}
Let $T_0$ be any $\Sigma_1$-sound recursively enumerable extension of $\IS$ and $\tau_0(v)$ be any $\Sigma_1$ definition of $T_0$. 
Let $\rho$ be a $\Pi_1$ Rosser sentence of $T_0$ defined by using $\tau_0(v)$. 
Also let $T_1 : = T_0 + \rho$ and $\tau_1(v)$ be $(\tau_0 + \rho)(v)$. 

1. Since $\IS \vdash \Con_{\tau_0} \leftrightarrow \neg \PR_{\tau_0}(\gdl{\neg \rho})$, $\IS \vdash \Con_{\tau_0} \leftrightarrow \Con_{\tau_1}$. 

2. Let $\psi$ be any $\Sigma_1$ sentence with $T_1 \vdash \psi$. 
Then $T_0 \vdash \neg \rho \lor \psi$. 
Since $T_0$ is $\Sigma_1$-sound, $\N \models \neg \rho \lor \psi$. 
Since $\N \models \rho$, $\N \models \psi$. 
Hence $T_1$ is $\Sigma_1$-sound. 

Moreover, since $\IS \vdash \Con_{\tau_0} \to \rho$, $T_0 + \Con_{\tau_0}$ is deductively equivalent to $T_0 + \rho + \Con_{\tau_0}$, and to $T_1 + \Con_{\tau_1}$. 

3. For any $\mathcal{L}_A$-formula $\varphi(\vec{x})$, 
\begin{align*}
	\IS \vdash \PR_{\tau_0}(\gdl{\Con_{\tau_0} \to \varphi(\vec{\dot{x}})}) & \leftrightarrow \PR_{\tau_0 + \Con_{\tau_0}}(\gdl{\varphi(\vec{\dot{x}})}), \\
& \leftrightarrow \PR_{\tau_0 + \rho + \Con_{\tau_0 + \rho}}(\gdl{\varphi(\vec{\dot{x}})}), \\
& \leftrightarrow \PR_{\tau_1}(\gdl{\Con_{\tau_1} \to \varphi(\vec{\dot{x}})}). 
\end{align*}

4. Since $T_1$ is $\Sigma_1$-sound and $T_0$ is consistent, $T_1 \nvdash \neg \Con_{\tau_0}$. 
%Since $T_1 \vdash \rho$, $T_1 \vdash \PR_{\tau_1}(\gdl{\rho})$. 
%Also $T_0 \vdash \Con_{\tau_0} \to \neg \PR_{\tau_0}(\gdl{\rho})$. 
%Suppose, towards a contradiction, $T_1 \vdash \neg \Con_{\tau_0}$. 
%Then $T_0 + \rho \vdash \neg \Con_{\tau_0}$. 
%Since $T_0 + \neg \rho \vdash \neg \Con_{\tau_0}$, we get $T_0 \vdash \neg \Con_{\tau_0}$. 
%This contradicts the $\Sigma_1$-soundness of $T_0$. 
%Thus $T_1 \nvdash \neg \Con_{\tau_0}$. 
It follows from Corollary \ref{QPL_characterization} that $\QPL_{\tau_0}(T_0) \nsubseteq \QPL_{\tau_1}(T_1)$. 
\end{proof}

\section{$\Sigma_1$ arithmetical interpretations}

In this section, we investigate inclusions between quantified provability logics with respect to $\Sigma_1$ arithmetical interpretations. 
The main goal of this section is to give a necessary and sufficient condition for the inclusion relation between quantified provability logics with respect to $\Sigma_1$ arithmetical interpretations.

\begin{defn}
An arithmetical interpretation $f$ is $\Sigma_n$ if for any atomic formula $P(\vec{x})$ of quantified modal logic, $f(P(\vec{x}))$ is a $\Sigma_n$ formula. 
\end{defn}

Notice that there are natural $\Sigma_1$ arithmetical interpretations. 
We introduce the quantified provability logics with respect to $\Sigma_n$ arithmetical interpretations. 

\begin{defn}
$\QPL_\tau^{\Sigma_n}(T) : = \{\varphi \mid \varphi$\ is a sentence and for all $\Sigma_n$ arithmetical interpretations $f$, $T \vdash f_\tau(\varphi)\}$. 
\end{defn}

Berarducci \cite{Ber89} proved that restricting arithmetical interpretations to $\Sigma_n$ does not change the complexity of quantified provability logics, that is, for each $n \geq 1$, the complexity of the quantified provability logic of $\PA$ with respect to $\Sigma_n$ arithmetical interpretations is also $\Pi^0_2$-complete. 

On the other hand, it is beneficial to deal with $\Sigma_1$ arithmetical interpretations in our study. 
In the proof of Artemov's Lemma, the assumption $\Con_\tau \land f_\tau(\D)$ is prepared to make the formulas $f(P_K(x))$ and $\neg f(P_K(x, y))$ equivalent to $\Sigma_1$ formulas for each $K \in \{Z, S, A, M, L, E\}$. 
In the case that $f$ is a $\Sigma_1$ arithmetical interpretation, the same result holds without the assumption $\Con_\tau \land f_\tau(\D)$ by adding sufficiently many theorems of $\IS$ to the sentence $\chi$ as conjuncts. 
This is guaranteed by the following equivalences: 
\begin{itemize}
	\item $\neg P_Z(x) \leftrightarrow \exists y P_S(y, x)$; 
	\item $\neg P_S(x, y) \leftrightarrow \exists z (P_S(x, z) \land (P_L(z, y) \lor P_L(y, z)))$; 
	\item $\neg P_A(x, y, z) \leftrightarrow \exists w (P_A(x, y, w) \land (P_L(w, z) \lor P_L(z, w)))$; 
	\item $\neg P_M(x, y, z) \leftrightarrow \exists w (P_M(x, y, w) \land (P_L(w, z) \lor P_L(z, w)))$; 
	\item $\neg P_L(x, y) \leftrightarrow P_E(x, y) \lor P_L(y, x)$; 
	\item $\neg P_E(x, y) \leftrightarrow P_L(x, y) \lor P_L(y, x)$.  
\end{itemize}
Thus we obtain the following variation of Artemov's Lemma with respect to $\Sigma_1$ arithmetical interpretations. 

\begin{thm}[$\Sigma_1$-Artemov's Lemma]
There exists an $\mathcal{L}_A$-sentence $\chi$ such that $\IS \vdash \chi$ and for any $\Sigma_1$ arithmetical interpretation $f$ and any $\mathcal{L}_A$-formula $\varphi(\vec{x})$, 
\[
	\IS \vdash f_\tau(\chi^\circ) \land R_f(\vec{x}, \vec{y}) \to (\varphi(\vec{x}) \leftrightarrow f_\tau(\varphi^\circ(\vec{y}))). 
\]
\qed
\end{thm}

We also obtain a variation of Fact \ref{BooSurj} with respect to $\Sigma_1$ arithmetical interpretations. 

\begin{prop}\label{BooSurj2}
For any $\Sigma_1$ arithmetical interpretation $f$,
\[
	\IS \vdash f_\tau(\chi^\circ) \to \forall y \exists x R_f(x, y). 
\]
\qed
\end{prop}

The following proposition is a variation of Fact \ref{VD} with respect to $\Sigma_1$ arithmetical interpretations. 

\begin{prop}\label{VD2}
For any $\mathcal{L}_A$-sentence $\varphi$, the following are equivalent: 
\begin{enumerate}
	\item $T \vdash \varphi$. 
	\item $\chi^\circ \to \varphi^\circ \in \QPL_\tau^{\Sigma_1}(T)$. 
\end{enumerate}
\end{prop}
\begin{proof}
$(1 \Rightarrow 2)$: Suppose $T \vdash \varphi$. 
By $\Sigma_1$-Artemov's Lemma, for any $\Sigma_1$ arithmetical interpretation $f$, $\IS \vdash f_\tau(\chi^\circ) \to (\varphi \leftrightarrow f_\tau(\varphi^\circ))$. 
Then $T$ proves $f_\tau(\chi^\circ \to \varphi^\circ)$.
Thus $\chi^\circ \to \varphi^\circ \in \QPL_\tau^{\Sigma_1}(T)$. 

$(2 \Rightarrow 1)$: Suppose $\chi^\circ \to \varphi^\circ \in \QPL_\tau^{\Sigma_1}(T)$. 
By considering a natural $\Sigma_1$ arithmetical interpretation, we obtain $T \vdash \varphi$. 
\end{proof}

We prove the following main theorem of this section. 

\begin{thm}\label{MT}
The following are equivalent: 
\begin{enumerate}
	\item $\QPL_{\tau_0}^{\Sigma_1}(T_0) \subseteq \QPL_{\tau_1}^{\Sigma_1}(T_1)$. 
	\item $\Th(T_0) \subseteq \Th(T_1)$ and for any $\mathcal{L}_A$-formula $\varphi(\vec{x})$, 
\[
	T_1 \vdash \forall \vec{x} (\PR_{\tau_0}(\gdl{\varphi(\vec{\dot{x}})}) \leftrightarrow 		\PR_{\tau_1}(\gdl{\varphi(\vec{\dot{x}})})).
\] 
\end{enumerate}
\end{thm}
\begin{proof}
$(1 \Rightarrow 2)$: Suppose $\QPL_{\tau_0}^{\Sigma_1}(T_0) \subseteq \QPL_{\tau_1}^{\Sigma_1}(T_1)$. 

First, we prove $\Th(T_0) \subseteq \Th(T_1)$. 
Let $\varphi$ be any sentence with $T_0 \vdash \varphi$. 
Then by Proposition \ref{VD2}, $\chi^\circ \to \varphi^\circ \in \QPL_{\tau_0}^{\Sigma_1}(T_0)$. 
By the supposition, this sentence is also in $\QPL_{\tau_1}^{\Sigma_1}(T_1)$. 
Then by Proposition \ref{VD2} again, we obtain $T_1 \vdash \varphi$. 
Therefore $\Th(T_0) \subseteq \Th(T_1)$. 

Secondly, we prove the $T_1$-provable equivalence of the two provability predicates. 
Let $\varphi(\vec{y})$ be any $\mathcal{L}_A$-formula. 
By $\Sigma_1$-Artemov's Lemma, for any $\Sigma_1$ arithmetical interpretation $f$, 
\[
	\IS \vdash f_{\tau_0}(\chi^\circ) \land R_f(\vec{x}, \vec{y}) \to (\varphi(\vec{x}) \leftrightarrow f_{\tau_0}(\varphi^\circ(\vec{y}))). 
\]
By Fact \ref{DC}, 
\begin{equation}\label{eq7}
	\IS \vdash f_{\tau_0}(\Box \chi^\circ) \land \PR_{\tau_0}(\gdl{R_f(\vec{\dot{x}}, \vec{\dot{y}})}) \to \left(\PR_{\tau_0}(\gdl{\varphi(\vec{\dot{x}})}) \leftrightarrow f_{\tau_0}(\Box \varphi^\circ(\vec{y}))\right). 
\end{equation}
By $\Sigma_1$-Artemov's Lemma again, 
\begin{equation}\label{eq8}
	\IS \vdash f_{\tau_0}(\chi^\circ) \land R_f(\vec{x}, \vec{y}) \to \left(f_{\tau_0}(\PR_{\tau_0}(\gdl{\varphi(\vec{\dot{y}})})^\circ) \leftrightarrow \PR_{\tau_0}(\gdl{\varphi(\vec{\dot{x}})})\right). 
\end{equation}

By Lemma \ref{SelfProver1}.2, $\IS \vdash R_f(\vec{x}, \vec{y}) \to \PR_{\tau_0}(\gdl{R_f(\vec{\dot{x}}, \vec{\dot{y}})})$. 
By combining this with (\ref{eq7}) and (\ref{eq8}), 
\[
	\IS \vdash f_{\tau_0}(\boxdot \chi^\circ) \land R_f(\vec{x}, \vec{y}) \to \left(f_{\tau_0}(\PR_{\tau_0}(\gdl{\varphi(\vec{\dot{y}})})^\circ) \leftrightarrow f_{\tau_0}(\Box \varphi^\circ(\vec{y}))\right). 
\]
Since $\vec{x}$ does not appear in the consequent of the formula, 
\[
	\IS \vdash f_{\tau_0}(\boxdot \chi^\circ) \land \exists \vec{x} R_f(\vec{x}, \vec{y}) \to \left(f_{\tau_0}(\PR_{\tau_0}(\gdl{\varphi(\vec{\dot{y}})})^\circ) \leftrightarrow f_{\tau_0}(\Box \varphi^\circ(\vec{y}))\right). 
\]
By Proposition \ref{BooSurj2}, $\IS \vdash f_{\tau_0}(\chi^\circ) \to \forall \vec{y} \exists \vec{x} R_f(\vec{x}, \vec{y})$. 
Then, 
\[
	\IS \vdash f_{\tau_0}(\boxdot \chi^\circ) \to \left(f_{\tau_0}(\PR_{\tau_0}(\gdl{\varphi(\vec{\dot{y}})})^\circ) \leftrightarrow f_{\tau_0}(\Box \varphi^\circ(\vec{y}))\right). 
\]
We obtain
\[
	\forall \vec{y} \left(\boxdot \chi^\circ \to \left(\PR_{\tau_0}(\gdl{\varphi(\vec{\dot{y}})})^\circ \leftrightarrow \Box \varphi^\circ(\vec{y})\right) \right) \in \QPL_{\tau_0}^{\Sigma_1}(T_0) \subseteq \QPL_{\tau_1}^{\Sigma_1}(T_1). 
\]
By considering a natural $\Sigma_1$ arithmetical interpretation, we conclude
\[
	T_1 \vdash \forall \vec{y} \left(\PR_{\tau_0}(\gdl{\varphi(\vec{\dot{y}})}) \leftrightarrow \PR_{\tau_1}(\gdl{\varphi(\vec{\dot{y}})}) \right). 
\]

$(2 \Rightarrow 1)$: Assume Clause 2 of the statement. 
Let $A$ be any element of $\QPL_{\tau_0}^{\Sigma_1}(T_0)$ and $f$ be any $\Sigma_1$ arithmetical interpretation. 
Then, $T_0 \vdash f_{\tau_0}(A)$. 
Since $\Th(T_0) \subseteq \Th(T_1)$, $T_1 \vdash f_{\tau_0}(A)$. 
By the assumption and Lemma \ref{eqQPL}, we have $T_1 \vdash f_{\tau_0}(A) \leftrightarrow f_{\tau_1}(A)$, and thus $T_1 \vdash f_{\tau_1}(A)$. 
Therefore $A$ is in $\QPL_{\tau_1}^{\Sigma_1}(T_1)$.  
We have proved $\QPL_{\tau_0}^{\Sigma_1}(T_0) \subseteq \QPL_{\tau_1}^{\Sigma_1}(T_1)$. 
\end{proof}

Similar to the proof of $(2 \Rightarrow 1)$ of Theorem \ref{MT}, it can be proved that Clause 2 in the statement of Theorem \ref{MT} implies $\QPL_{\tau_0}(T_0) \subseteq \QPL_{\tau_1}(T_1)$.

\begin{cor}\label{CorS1}
If $\QPL_{\tau_0}^{\Sigma_1}(T_0) \subseteq \QPL_{\tau_1}^{\Sigma_1}(T_1)$, then $\QPL_{\tau_0}(T_0) \subseteq \QPL_{\tau_1}(T_1)$. 

\qed
\end{cor}

We propose the following question. 

\begin{prob}
Does the converse implication of Corollary \ref{CorS1} hold?
\end{prob}

We close this section with the following corollary. 

\begin{cor}
If $\QPL_{\tau_0}^{\Sigma_1}(T_0) \subseteq \QPL_{\tau_1}^{\Sigma_1}(T_1)$ and $T_1$ is $\Sigma_1$-sound, then $\QPL_{\tau_0}^{\Sigma_1}(T_0) = \QPL_{\tau_1}^{\Sigma_1}(T_1)$. 
\end{cor}
\begin{proof}
Suppose $\QPL_{\tau_1}^{\Sigma_1}(T_0) \subseteq \QPL_{\tau_1}^{\Sigma_1}(T_1)$ and $T_1$ is $\Sigma_1$-sound. 
By Theorem \ref{MT}, $T_1 \vdash \forall \vec{x} (\PR_{\tau_0}(\gdl{\varphi(\vec{\dot{x}})}) \leftrightarrow \PR_{\tau_1}(\gdl{\varphi(\vec{\dot{x}})}))$ for any $\mathcal{L}_A$-formula $\varphi(\vec{x})$. 
Let $\psi$ be any $\mathcal{L}_A$-sentence with $T_1 \vdash \psi$. 
Since $T_1$ proves $\PR_{\tau_1}(\gdl{\psi})$ by Fact \ref{DC}.1, we have $T_1 \vdash \PR_{\tau_0}(\gdl{\psi})$. 
Since $T_1$ is $\Sigma_1$-sound, $T_0 \vdash \psi$. 
We have shown $\Th(T_1) \subseteq \Th(T_0)$. 
Then, for any $\mathcal{L}_A$-formula $\varphi(\vec{x})$, $T_0 \vdash \forall \vec{x}(\PR_{\tau_0}(\gdl{\varphi(\vec{\dot{x}})}) \leftrightarrow \PR_{\tau_1}(\gdl{\varphi(\vec{\dot{x}})}))$. 
By Theorem \ref{MT}, we conclude $\QPL_{\tau_1}^{\Sigma_1}(T_1) \subseteq \QPL_{\tau_0}^{\Sigma_1}(T_0)$. 
\end{proof}

\bibliographystyle{plain}
\bibliography{ref}

\end{document}